\documentclass{amsart}
\usepackage{amssymb, enumerate}

\newtheorem{theorem}{Theorem}[section]

\theoremstyle{definition}
\newtheorem{definition}[theorem]{Definition}
\newtheorem{corollary}[theorem]{Corollary}
\newtheorem{example}[theorem]{Example}

\theoremstyle{remark}

\numberwithin{equation}{section}

\newcommand{\sgn}{\operatorname{sgn}}

\begin{document}

\title[Modular Inverse and Chinese Remainder Algorithm]{Modular Inverses and\\ Chinese Remainder Algorithm}


\author{W. H. Ko}
\address{}
\curraddr{Park Avenue, Mongkok, HKSAR, China.}
\email{wh\_ko@hotmail.com}
\thanks{}



\date{}

\dedicatory{}

\begin{abstract}
This paper introduces two forms of modular inverses and proves their reciprocity formulas respectively.  These formulas are then applied to formulate new and generalized algorithm for computing these modular inverses.  The same algorithm is also shown to be applicable for the Chinese Remainder problem, i.e., simultaneous linear congruence equations, for co-prime moduli as well as non-co-prime moduli.
\end{abstract}

\maketitle


\section{Introduction}
Modular inverse is one of the basic operations in modular arithmetic, and it is applied extensively in computer science and telecommunications, particularly, in cryptography. However, it is also a time-consuming operation implemented in hardware or software compared with other modular arithmetic operations such as addition, subtraction and even multiplication. Efficient algorithms in calculating modular inverse have been sought after for the past few decades and any improvements will still be welcome.

In implementing an efficient algorithm to calculate the modular inverse of the form $b^{-1} \pmod{2^m}$, Arazi and Qi \cite{AraziQi} have made use of a reciprocity trick, in Algorithm 3b, which can be translated mathematically to Eq.\eqref{eq:reciprocity}.

Although this reciprocity formula, Eq.\eqref{eq:reciprocity}, can be regarded as a modified form of the linear Diophantine equation Eq.\eqref{eq:diophan}, this formula as written in the form of Eq. \eqref{eq:reciprocity} is able to bring more insight into the properties of modular inverse, and thus allowing us formulate new algorithms to calculate modular inverse and to solve the Chinese Remainder problem efficiently, as we are going to show in the following discussions.

In a recent attempt \cite{zhangs} to modernize some ancient Chinese algorithms, Zhang and Zhang have introduced a reciprocity formula similar to Eq. \eqref{eq:reciprocity}.  However, its definition of $f_{a,b}$ and $f_{b,a}$ is slightly different from Eq. \eqref{eq:invdef} and thus its reciprocity formula is different.

The classical definition of modular inverse is reproduced in this section, and a modified definition of modular inverse is introduced in the next section. In Section 3, a new, extended modular inverse will be introduced.  Reciprocity formulas for these modular inverses will also be proved.

In the end of this paper, examples will be provided to illustrate the use of the generalized modular inverse algorithm to calculate the classical and extended modular inverses as well as the Chinese Remainder problem with co-prime moduli and non-co-prime moduli.

\begin{definition}
In modular arithmetic, the classical definition of modular inverse of an integer $a$ modulo $m$ is an integer $x$ such that
\begin{equation} \label{eq:congruence}
ax\equiv 1 \pmod{m}
\end{equation}
The modular inverse is generally denoted as
\begin{equation*}
x \equiv a^{-1} \pmod{m}
\end{equation*}
\end{definition}

Finding the modular inverse is equivalent to finding the solution of the following linear Diophantine equation, where $a,x,k,m \in\mathbb{Z}$,
\begin{equation}\label{eq:diophan}
a\,x-k\,m=1
\end{equation}

\section{Modular Inverse}

\subsection{Modular Inverse Definition}

Slightly different version of modular inverse is used throughout this discussion and it is introduced in this section. This definition of modular inverse will still satisfy the same congruence equation, Eq.\eqref{eq:congruence}.

Furthermore, in order to have a nicer presentation in the equations containing modular inverse,
new notations for modular operation and modular inverse will be used within this paper.

\begin{definition}
For all $a, m \in\mathbb{Z}$, a modulo m, denoted by $(a)_{m}$, is defined as :
\begin{equation}\label{eq:moddef}
(a)_{m}=a-m \bigg\lfloor\dfrac{a}{m}\bigg\rfloor
\end{equation}
Note that \\
\indent \quad \quad
$\begin{cases} 
0 \le (a)_{m} < m & \text{if }m>0\\
m < (a)_{m} \le 0 & \text{if }m<0
\end{cases}$
\end{definition}

\begin{definition}
Let $a, m \in\mathbb{Z}\setminus \{0\} \text{ and } \gcd(a,m)=1$, modular inverse $a$ modulo $m$, denoted by $(a^{-1})_{m}$, is defined as :
\begin{equation}\label{eq:invdef}
(a^{-1})_{m}=x,
\end{equation}
\begin{equation*}
where \begin{cases}
1\le x\le m-1 & \text{if } 1<m \,\&\, a\:x \equiv 1 \pmod{m} \\
m+1\le x \le-1 & \text{if } m<-1 \,\&\, a\,x \equiv 1 \pmod{m} \\
x=\dfrac{1}{2}|m|(\sgn(m)-\sgn(a))+\sgn(a) & \text{if } |m|=1 \\
Undefined & \text{if } a\,m=0 \text{ or } \gcd(a,m) \ne 1
\end{cases}\notag
\end{equation*}
\end{definition}

Obviously, for $|m|>1, (a^{-1})_{m}$ satisfies the congruence requirement that
 $$a(a^{-1})_{m} \equiv 1 \pmod{m}$$.

For $|m|=1$, there are two cases.
\begin{description}
\item[Case 1]$m=1\\
\begin{aligned}
(a^{-1})_{m}&=\dfrac{1}{2}|1|(\sgn(1)-\sgn(a))+\sgn(a)=\dfrac{1}{2}(1-\sgn(a))+\sgn(a)\notag\\
&=\begin{cases} 1 & \text{if } a>0 \\ 0 & \text{if } a<0 \end{cases}\notag
\end{aligned}$
\item[Case 2]
$m=-1\\
\begin{aligned}
(a^{-1})_{m}&=\dfrac{1}{2}|-1|(\sgn(-1)-\sgn(a))+\sgn(a)=\dfrac{1}{2}(-1-\sgn(a))+\sgn(a)\notag\\&=\begin{cases} 0 & \text{if } a>0 \\ -1 & \text{if } a<0 \end{cases}\notag.
\end{aligned}$
\end{description}

Hence for $|m|=1, (a^{-1})_{m}$ also satisfies the requirement that $a(a^{-1})_{m} \equiv 1 \pmod{m}$.

Particularly, note that
\begin{equation*}
(1^{-1})_m=(1)_m=
\begin{cases}
1 & m \ge 1\\
m+1 & m \le -1
\end{cases}
\end{equation*}
\begin{equation*}
(-1^{-1})_m=(-1)_m=
\begin{cases}
m-1 & m \ge 1\\
-1 & m \le -1
\end{cases}
\end{equation*}

It is also interesting to note that the modular inverse for $|m|=1$ as defined by Eq.\eqref{eq:invdef} 
is slightly different from the classical definition that $a^{-1} \pmod{m} =0$ for $|m|=1$ and non-zero $a$.

\subsection{Reciprocity Formula}

The reciprocity relationship between modular inverses seems obvious from the linear Diophantine equation, Eq.\eqref{eq:diophan} and, as a matter of fact, this is the Euclidean algorithm (i.e., iterative division) in disguise. However, this reciprocity formula is not found in any classic text such as \cite{hardy+wright}, nor in any modern text, such as \cite{ireland+rosen, song}. The reciprocity identity first appeared in \cite{joye+paillier} as Lemma 1 in a format similar to Eq.\eqref{eq:reciprocity}. However, only positive integers were discussed for specific type of cryptography applications and the conditions that $|m|=1$ was not taken care of. By the way, on the footnote of p.244 of \cite{joye+paillier}, it stated that Arazi was "the first to take advantage of this folklore theorem to implement fast modular inversions".

\begin{theorem}\label{th:reciprocity}\emph{(Reciprocity formula)}
Let $a, b \in\mathbb{Z} \text{ and } \gcd(a,b)=1$, then 

\begin{equation}\label{eq:reciprocity}
a\,(a^{-1})_{b}+b\,(b^{-1})_{a}=1+a \, b. 
\end{equation} 
\end{theorem}

\begin{proof}

\begin{description}
\item[Case 1]$a>1,b>1$\\
Let $U=a(a^{-1})_{b}+b(b^{-1})_{a}$.  Since $U\equiv1\pmod{a}$ and $U\equiv1\pmod{b}$, then $U\equiv1\pmod{a b}.$
That is, $U=1+kab, \text{ where k }\in\mathbb{Z}$. Therefore,\\
$1<a+b \le U=1+k\,a\,b \le a(b-1)+b(a-1)=2a\,b-(a+b)<2a\,b \implies 0<k\,a\,b<2a\,b \implies 0<k<2 \implies k=1$.\\
Therefore $a\,(a^{-1})_{b}+b\,(b^{-1})_{a}=1+a\,b.$

\item[Case 2]$a>1,b<-1$\\
Since $b+1\le (a^{-1})_{b} \le -1 \implies a(b+1) \le a(a^{-1})_{b} \le -a$, and $1 \le (b^{-1})_{a} \le a-1 \implies b(a-1) \le b(b^{-1})_{a} \le b$, therefore, 
$a(b+1) + b(a-1) \le a(a^{-1})_{b} + b(b^{-1})_{a} \le -a+b \implies 2a\,b-(a-b) \le U=1+k\,a\,b \le -(a-b) \implies 2a\,b<2a\,b-(a+1-b) \le k\,a\,b \le -(a+1-b)<0 \implies 0<k<2 \implies k=1$.

\item[Case 3]$a<-1,b>1$\\
Similar to Case 2, and therefore $k=1$.

\item[Case 4]$a<-1,b<-1$\\
Since $b+1 \le (a^{-1})_{b} \le -1 \implies -a \le a(a^{-1})_{b} \le a(b+1)$, and $a+1 \le (b^{-1})_{a} \le -1 \implies -b \le b(b^{-1})_{a} \le b(a+1)$,
therefore $ -a-b \le a(a^{-1})_{b}+b(b^{-1})_{b} \le a(b+1)+b(a+1) \implies -(a+b) \le U=1+k\,a\,b \le 2a\,b+(a+b) \implies 0<-(a+b+1) \le k\,a\,b \le 2a\,b+(a+b+1) < 2a\,b \implies 0<k<2 \implies k=1.$

\item[Case 5]$a=1,b>1$\\
Since $(a^{-1})_b = 1$ and $(b^{-1})_a = 1$, hence, $a(a^{-1})_{b}+b(b^{-1})_{a}= 1 \cdot 1+b \cdot 1=1+b=1+a\,b$.

\item[Case 6]$a=1,b<-1$\\
Since $(a^{-1})_b = b+1$ and $(b^{-1})_a = 0$, hence, $a(a^{-1})_{b}+b(b^{-1})_{a}= 1(b+1)+b \cdot 0=1+b=1+a\,b$.

\item[Case 7]$a=-1,b>1$\\
Since $(a^{-1})_b = 0$ and $(b^{-1})_a = 0$, hence, $a(a^{-1})_{b}+b(b^{-1})_{a}= (-1) \cdot 0+b \cdot 0=0=1+a\,b$.

\item[Case 8]$a=-1,b<-1$\\
Since $(a^{-1})_b = -1$ and $(b^{-1})_a = -1$, hence, $a(a^{-1})_{b}+b(b^{-1})_{a}= (-1) \cdot (-1)+b \cdot (-1)=1-b=1+a\,b$.

\item[Case 9]$|a|=1 \text{ and } |b|=1$\\
$a(a^{-1})_{b}+b(b^{-1})_{a}=a(\frac{1}{2}|b|(\sgn(b)-\sgn(a))+\sgn(a))+b(\frac{1}{2}|a|(\sgn(a)-\sgn(b))+\sgn(b))=b(a+\sgn(b))+\sgn(a)(a-a\, b\, \sgn(b))
=(1+a\,b)-(1-a \,\sgn(a))(1-b \,\sgn(b))=(1+a\,b)-(1-|a|)(1-|b|)=1+a\,b.$
\end{description}
\end{proof}

\begin{corollary}
If $a,b,k \in\mathbb{Z} \text{ and } \gcd(a,b)=1$, then\\
$((k\,a+b)^{-1})_{a}=\begin{cases} (b^{-1})_{a} & |a|>1\\
(b^{-1})_{a}+\frac{1}{2}(\sgn(k\,a+b)-\sgn(b))) & |a|=1
\end{cases}$
\end{corollary}

\begin{proof}
\begin{description}
\item[Case 1] $|a|>1$\\
$a(a^{-1})_{b}+b(b^{-1})_{a}=1+a\,b \implies (k\,a+b)(b^{-1})_{a}=1+a(b-(a^{-1})_{b}+k(b^{-1})_{a}) \implies
((k\,a+b)^{-1})_{a}=(b^{-1})_{a}$

\item[Case 2] $|a|=1$\\
$((k\,a+b)^{-1})_{a}-(b^{-1})_{a}\\
=\frac{1}{2}|a|(\sgn(a)-\sgn(k\,a+b))+\sgn(k\,a+b)-(\frac{1}{2}|a|(\sgn(a)-\sgn(b))+\sgn(b))\\
=\frac{1}{2}(|a|-2)(\sgn(b)-\sgn(k\,a+b))=\frac{1}{2}(\sgn(k\,a+b)-\sgn(b))$
\end{description}
\end{proof}

\begin{corollary}
If $a,b,k \in\mathbb{Z}, |a| > 1$ and $\gcd(a,b)=1$, then
\begin{equation}\label{eq:invform}
\begin{aligned}
(a^{-1})_{k\,a+b}&=k(a-(b^{-1})_{a})+(a^{-1})_{b}\\
(a^{-1})_{k\,a-b}&=k(b^{-1})_{a}-(b-(a^{-1})_{b})
\end{aligned}
\end{equation}
\end{corollary}

\begin{proof}
Let $c=k\,a+b$, and note $a\,(a^{-1})_{c}+c\,(c^{-1})_{a}=1+a \, c.$
\begin{align}
(a^{-1})_{k\,a+b}&=\dfrac{1+a(k\,a+b)-(k\,a+b)((k\,a+b)^{-1})_{a}}{a}\notag\\
&=(k\,a+b)+\dfrac{1-(k\,a+b)(b^{-1})_{a}}{a}\notag\\
&=(k\,a+b)-k(b^{-1})_{a}+\dfrac{1-b(b^{-1})_{a}}{a}\notag\\
&=k(a-(b^{-1})_{a})+(a^{-1})_{b}\notag
\end{align}

This completes the first equation.  And note $((k\,a-b)^{-1})_a=((-b)^{-1})_a=a-(b^{-1})_a$, and the second equation will be obtained.
\end{proof}

It is interesting to note that Eq.\eqref{eq:invform} fails if the classic definition of $(a^{-1})_b=0 \text{ for }|b|=1$ is used.

\section{Extended Modular Inverse}

\subsection{Extended Modular Inverse Definition}
We are also interested to find $x$ satisfying
\begin{equation}\label{eq:extcongruence}
ax\equiv b \pmod{m}
\end{equation}
and the solution is well-known.
For example, in \cite[p.125]{song}, the solution is given as
\begin{equation}\label{eq:extmodinv}
x\equiv b \cdot \dfrac{1}{a} \pmod{m}
\end{equation}
provided that $\gcd(a,m)=1$.  The solution provided by Eq.\eqref{eq:extmodinv} contains all the integers congruent to $x\pmod{m}$.

A new function, namely, extended modular inverse is introduced as follows.
\begin{definition}
Let $a,b,m \in\mathbb{Z}$ and $\gcd(a,m)=d$,extended modular inverse, denoted as $(b(a^{-1})_m)_m$, is defined as
\begin{equation}\label{eq:extmoddef}
(b(a^{-1})_m)_m =\begin{cases} (x)_m & d=1 \,\&\,  a\,x\equiv b\pmod{m}\\
(x)_{\frac{m}{d}} & d\ne1 \,\& \, d|b \,\&\, \left(\dfrac{a}{d}\right)x\equiv\left(\dfrac{b}{d}\right) \pmod{\left(\dfrac{m}{d}\right)}\\
\text{Undefined} & d \nmid b
 \end{cases}
\end{equation}
\end{definition}
Extended modular inverse is having the following properties :
\begin{enumerate}[i.]
\item
If $\gcd(a,m)=1$, then $x=(b(a^{-1})_m)_m$ satisfies the linear congruence equation, Eq.\eqref{eq:extcongruence}.\\
This is obvious from the definition above.

\item 
If $\gcd(a,m)=d\ne1, d|b$, then
\begin{equation}\label{eq:coprimesol}
x_i=\left(\left(\dfrac{b}{d}\right)\left(\dfrac{a}{d}\right)^{-1}\right)_{\frac{m}{d}}+i\dfrac{m}{d}, \qquad i=0,1, \ldots ,d-1
\end{equation}
satisfies the linear congruence equation, Eq.\eqref{eq:extcongruence}.

\begin{flalign}
a\,x_i&=a\left(\left(\left(\dfrac{b}{d}\right)\left(\dfrac{a}{d}\right)^{-1}\right)_{\frac{m}{d}}+i\dfrac{m}{d}\right)=d\left(\dfrac{a}{d}\right)\left(\left(\dfrac{b}{d}\right)\left(\dfrac{a}{d}\right)^{-1}\right)_{\frac{m}{d}}+m\dfrac{i\,a}{d}\notag\\
&=d\left(\dfrac{b}{d}+k\dfrac{m}{d}\right)+m\dfrac{i\,a}{d}=b+m\left(k+i\dfrac{a}{d}\right) , \text{where } k \in \mathbb{Z}. \notag
\end{flalign}
Hence, $a\,x_i \equiv b \pmod{m}$.

\item If, $d\mid b$, then by Eq.\eqref{eq:moddef},
\begin{equation}
\begin{cases}
0 \le (b(a^{-1})_m)_m < m & 0<m\\
m<(b(a^{-1})_m)_m \le 0 & m<0
\end{cases}\notag
\end{equation}

\item If $\gcd(a,m)=1, \gcd(a,b)=g$, then
\begin{equation}\label{eq:coprimemodinv}
(b(a^{-1})_m)_m=\left(\dfrac{b}{g}\left(\dfrac{a}{g}\right)^{-1}_m\right)_m
\end{equation}
Let $x=\left(\dfrac{b}{g}\left(\dfrac{a}{g}\right)^{-1}_m\right)_m$.\\
$a\,x=g\left(\dfrac{a}{g}\right)\left(\dfrac{b}{g}\left(\dfrac{a}{g}\right)^{-1}_m\right)_m=g\left(\dfrac{b}{g}+k\,m\right)=b+k\,g\,m, \text{where } k \in \mathbb{Z}.\\$
Hence, $x$ satisfies Eq.\eqref{eq:extcongruence} and $(b(a^{-1})_m)_m=x$.
\end{enumerate}

\subsection{Reciprocity Formula for Extended Modular Inverse}

Extended Euclidean algorithm, rooted back to over 2000 years ago, is still the most basic algorithm for calculating the ordinary modular inverse, i.e., $(a^{-1})_m$.

On the other hand, in ancient China, Dayan Qiuyi Shu (Dayan Algorithm to Find One) was used to calculate the modular inverse, as noted by Ding et al. on \cite[p.20]{ding+pei+salomaa}.
Both the Euclidean algorithm and Dayan Algorithm use the trick of continuing divisions to obtain pairs of smaller numbers until the solution is obtained.
However, the Euclidean algorithm stops at zero while the Dayan algorithm stops when it reaches one, i.e., $r_n=1$, since it is assumed that the two numbers are co-prime to each other.
And because of this characteristic, the Chinese name of this algorithm states that its objective is to "find one".

A straightforward means to calculate extended modular inverse is to first calculate the ordinary modular inverse, and then multiply by $b$, followed by a modulo operation to obtain the remainder,
as Eq. \eqref{eq:extmodinv} shows. This is the most obvious way by means of the definition of the extended modular inverse, Eq.\eqref{eq:extmoddef}.

However, the following section is to introduce a different algorithm which is also based on repeated division as Euclidean algorithm and Dayan algorithm do.

For the new algorithm, the following Reciprocity formula is needed.

\begin{theorem}\label{th:reciprocity2}\emph{(Reciprocity Formula for Extended Modular Inverse)}\\
Let $p,q,a \in\mathbb{Z}$, $0<q<p,0 \le a <p, \gcd(p,q)=1$, and let $p=c\,q+s\,r, a=\beta\,q-s\,\gamma$, 
where $0 \le r<q, 0\le\gamma<q,s\in\{-1,1\}$, then
\begin{equation}\label{eq:extreciprocity}
(a(q^{-1})_p)_p=\begin{cases}
\dfrac{a}{q}+\dfrac{p}{q}(\gamma(r^{-1})_q)_q=\dfrac{a}{q}+\dfrac{p}{q}(-(s\,a)_q((s\,p)^{-1})_q)_q & 1<q\\
a & q=1
\end{cases}
\end{equation}
\end{theorem}

\begin{proof}
\begin{enumerate}[1.]
\item $q=1$\\
Since $(1^{-1})_p=1$ and $0\le a<p$, therefore $(a(1^{-1})_p)_p=(a)_p=a$.
\item $s=1$

Note $a=\beta\,q-\gamma$, where $\beta, \gamma \in\mathbb{Z}$, and $0 \le \gamma < q$.
With Eq.\eqref{eq:reciprocity} and Eq.\eqref{eq:invform}, 
\begin{enumerate}[i)]
\item
$1<q \text{ and }(q^{-1})_{c\,q+r}=c(q-(r^{-1})_q)+(q^{-1})_r \text{ and } q(q^{-1})_r+r(r^{-1})_q=1+q\,r\\
\implies (a(q^{-1})_p)_p=((\beta\,q-\gamma)(c(q-(r^{-1})_q)+(q^{-1})_r))_p\\
=(\beta(c\,q^2-c\,q(r^{-1})_q+q(q^{-1})_r)-\gamma(c\,q-c(r^{-1})_q+(q^{-1})_r)))_p\\
=(\beta(1+(c\,q+r)(q-(r^{-1})_q))-\gamma(c\,q-c(r^{-1})_q+(q^{-1})_r)))_p\\
=\left(\beta-\gamma\left(c\,q-c(r^{-1})_q+\dfrac{1+q\,r-r(r^{-1})_q}{q}\right)\right)_p\\
=\left(\beta-\gamma\left(\dfrac{1+(c\,q+r)(q-(r^{-1})_q)}{q}\right)\right)_p
=\left(\beta-\dfrac{\gamma}{q}+\dfrac{p}{q}\gamma(r^{-1})_q\right)_{p}\\
=\left(\beta-\dfrac{\gamma}{q}+\dfrac{p}{q}\left(q \bigg \lfloor \dfrac{\gamma(r^{-1})_q}{q} \bigg \rfloor+(\gamma(r^{-1})_q)_q\right)\right)_{p}\\
=\left(\beta-\dfrac{\gamma}{q}+\dfrac{p}{q}(\gamma(r^{-1})_q)_q\right)_p=\left(\dfrac{a}{q}+\dfrac{p}{q}(\gamma(r^{-1})_q)_q\right)_p$
\item
$0 \le a\le p-1 \implies 0\le \dfrac{a}{q}+\dfrac{p}{q}(\gamma(r^{-1})_q)_q \le \dfrac{p-1}{q}+\dfrac{p}{q}(q-1)=\dfrac{p\,q-1}{q}<p$
\item
From i) and ii), we can get, for $1<q$,
\begin{flalign}
\left(\dfrac{a}{q}+\dfrac{p}{q}(\gamma(r^{-1})_q)_q\right)_p&=\dfrac{a}{q}+\dfrac{p}{q}(\gamma(r^{-1})_q)_q\notag\\
&=\dfrac{a}{q}+\dfrac{p}{q}((\beta\,q-a)_q)((c\,q+p)^{-1})_q)_q\notag\\
&=\dfrac{a}{q}+\dfrac{p}{q}(-(a)_q(p^{-1})_q)_q\notag
\end{flalign}
This completes the proof for the $s=1$.
\end{enumerate}

\item $s=-1$\\
Note that $a=\beta\,q+\gamma$ and, from Eq.\eqref{eq:invform}, $(q^{-1})_{c\,q-r}=c\,(r^{-1})_q-(r-(q^{-1})_r)$.
\begin{enumerate}[i.)]
\item
\begin{flalign}
(a(q^{-1})_p)_p&=((\beta\,q+\gamma)(q^{-1})_{c\,q-r})_{c\,q-r}\notag\\
&=((\beta\,q+\gamma)(c\,(r^{-1})_q-(r-(q^{-1})_r)))_{c\,q-r}\notag\\
&=(c(\beta\,q+\gamma)(r^{-1})_q-r(\beta\,q+\gamma)+(q^{-1})_r))(\beta\,q+\gamma)_{c\,q-r}\notag\\
&=\left(\beta+\dfrac{\gamma}{q}\left(q(q^{-1})_r-q\,r\right)+c\,\gamma(r^{-1})_q+\beta(c\,q-r)(c^{-1})_q\right)_{c\,q-r}\notag\\
&=\left(\beta+\dfrac{\gamma}{q}\left(q(q^{-1})_r-q\,r\right)+c\,\gamma(r^{-1})_q\right)_{c\,q-r}\notag\\
&=\left(\beta+\dfrac{\gamma}{q}\left(1-r(r^{-1})_q\right)+c\,\gamma(r^{-1})_q\right)_{c\,q-r}=\left(\dfrac{a}{q}+\dfrac{p}{q}\gamma(r^{-1})_q\right)_p\notag\\
&=\left(\dfrac{a}{q}+\dfrac{p}{q}\left( q\bigg\lfloor \dfrac{\gamma(r^{-1})_q}{q} \bigg\rfloor+\left(\gamma(r^{-1})_q\right)_q\right) \right)_p\notag\\
&=\left(\dfrac{a}{q}+\dfrac{p}{q}\left(\gamma(r^{-1})_q\right)_q \right)_p=\dfrac{a}{q}+\dfrac{p}{q}\left(\gamma(r^{-1})_q\right)_q \notag
\end{flalign}
\item
\begin{flalign}
\dfrac{a}{q}+\dfrac{p}{q}\left(\gamma(r^{-1})_q\right)_q&=\dfrac{a}{q}+\dfrac{p}{q}\left((a-\beta q)((c\,q-p)^{-1})_q\right)_q\notag\\
&=\dfrac{a}{q}+\dfrac{p}{q}(-(-a)_q((-p)^{-1})_q)_q\notag
\end{flalign}
\end{enumerate}
\end{enumerate}
\end{proof}
Obviously, Eq. \eqref{eq:extreciprocity} reduces to, Eq. \eqref{eq:reciprocity}, the Reciprocity formula for ordindary modular inverse for positive integers $p, q$, when $a=1$, since $(-(p^{-1})_q)_q=q-(p^{-1})_q$.

\begin{corollary}
Let $p,q,a \in\mathbb{Z}$, $0<q<p,0 \le a <p, \gcd(p,q)=d\ne1$, and let $p=c\,q+s\,r, a=\beta\,q-s\,\gamma$, 
where $0 \le r<q, 0\le\gamma<q,s\in\{-1,1\}$, then
\begin{equation*}
(a(q^{-1})_p)_p=\begin{cases} \dfrac{a}{q}+\dfrac{p}{q}(\gamma(r^{-1})_q)_q=\dfrac{a}{q}+\dfrac{p}{q}(-(s\,a)_q((s\,p)^{-1})_q)_q & d|a \\
\text{Undefined} & d\nmid a\end{cases}
\end{equation*}
\end{corollary}

\begin{proof}
Assume $d|a$, let $p'=\dfrac{p}{d}, q'=\dfrac{q}{d},a'=\dfrac{a}{d}$, then
\begin{enumerate}[1.]
\item
$p=c\,q+s\,r, a=\beta\,q-s\,\gamma \implies p'=c\,q'+s\,r', a'=\beta\,q'-s\,\gamma'$, and 
\item Pick up the smallest $x_i$ from Eq.\eqref{eq:coprimesol} and note Eq.\eqref{eq:extreciprocity},
\begin{flalign}
(a(q^{-1})_p)_p&=\left(\left(\dfrac{a}{d}\right)\left(\dfrac{b}{d}\right)^{-1}_{\frac{p}{d}}\right)_{\frac{p}{d}}=(a'(q'^{-1})_{p'})_{p'}
=\dfrac{a'}{q'}+\dfrac{p'}{q'}(\gamma'(r'^{-1})_{q'})_{q'}\notag\\
&=\dfrac{a'}{q'}+\dfrac{p'}{q'}((d\,\gamma')((d\,r')^{-1})_{d\,q'})_{d\,q'}=\dfrac{a}{q}+\dfrac{p}{q}(\gamma(r^{-1})_q)_q\notag
\end{flalign}
\end{enumerate}
\end{proof}

\begin{theorem}\label{th:theorem0}\emph{(Extended Modular Inverse Formula)}
\footnote{As Theorem \ref{th:theorem0} reduces to the Dayan algorithm for specific values, it may be called a generalized Dayan algorithm.  In fact, this algorithm was first inspired by my study on the Dayan algorithm\cite{man}.}
\\
Let $p,q,a,k_i,r_i, c_i, \beta_i, \gamma_i \in\mathbb{N_+}, 1<q<p, 0\le a<p, \gcd(p,q)=1$.\\
If $r_{i-1}=c_{i+1}\,r_i+s_{i+1}\,r_{i+1}, \gamma_i=\beta_i\,r_i-s_{i+1}\,\gamma_{i+1}, f_i=c_if_{i-1}+s_{i-1}f_{i-2}$,
where $r_{-1}=p,r_0=q,\gamma_0=(a)_p, f_0=1,f_1=c_1, 0\le r_{i+1}<r_i, 0\le\gamma_{i+1}<r_i$,
then
\begin{equation}\label{eq:extmodinvform1}
(a(q^{-1})_p)_p=\displaystyle\sum_{i=0}^{n} \dfrac{p\,\gamma_i}{r_{i-1}r_i},
\footnote{Eq.\ref{eq:extmodinvform1} can be viewed as a generalized form of the following formula deduced by Euler, $z=q+a\,b\,v\left(\dfrac{1}{ab}-\dfrac{1}{bc}+\dfrac{1}{cd}-\dfrac{1}{de}+ \ldots\right)$ \cite[p.61]{dickson}.}
\end{equation}
, and
\begin{equation}\label{eq:extmodinvform2}
(a(q^{-1})_p)_p=\displaystyle\sum_{i=0}^{n} f_i\beta_i
\end{equation}
where $n$ is such that $r_n=1$, or $\gamma_{n+1}=0$.\\
Note that this algorithm can be further generalized if the following formulas are used instead : $r_{i-1}=c_{i+1}\,r_i+s_{i+1,1}\,r_{i+1}, \gamma_i=\beta_i\,r_i+s_{i+1,2}\,\gamma_{i+1}$.
\end{theorem}

\begin{proof} Note $q>1, (a(q^{-1})_p)_p=(\gamma_0(r_0^{-1})_{r_{-1}})_{r_{-1}}$, and Eq. \eqref{eq:extreciprocity}.
\begin{enumerate}[1.)]
\item
Assume $r_i\ne1$ and $\gamma_{i+1}\ne0$ for all $i=0, 1, 2, \ldots, n-1$, then we can prove, by induction that
\begin{equation}\label{eq:extmodinvform3}
(\gamma_0(r_0^{-1})_{r_{-1}})_{r_{-1}}=\displaystyle\sum_{i=0}^{n} \dfrac{r_{-1}\gamma_i}{r_{i-1}r_i}+\dfrac{r_{-1}}{r_n}(\gamma_{n+1}(r_{n+1}^{-1})_{r_{n}})_{r_{n}}
\end{equation}
\begin{enumerate}[a.)]
\item
Since $(\gamma_0(r_0^{-1})_{r_{-1}})_{r_{-1}}=\dfrac{\gamma_0}{r_0}+\dfrac{r_{-1}}{r_0}(\gamma_1(r_1^{-1})_{r_{0}})_{r_{0}}
$
therefore, Eq. \eqref{eq:extmodinvform3} is true when $n=0$.
\item
If $r_i\ne1$ and $\gamma_{i+1}\ne0$ for all $i=0, 1, 2, \ldots ,k-1$, then\\
$(\gamma_0(r_0^{-1})_{r_{-1}})_{r_{-1}}=\displaystyle\sum_{i=0}^{k} \dfrac{r_{-1}\gamma_i}{r_{i-1}r_i}+\dfrac{r_{-1}}{r_k}(\gamma_{k+1}(r_{k+1}^{-1})_{r_{k}})_{r_{k}}$.\\
If $r_k\ne1$, and $\gamma_{k+1}\ne0$, then\\
$(\gamma_0(r_0^{-1})_{r_{-1}})_{r_{-1}}=\displaystyle\sum_{i=0}^{k} \dfrac{r_{-1}\gamma_i}{r_{i-1}r_i}+\dfrac{r_{-1}}{r_k}(\gamma_{k+1}(r_{k+1}^{-1})_{r_{k}})_{r_{k}}\\
=\displaystyle\sum_{i=0}^{k} \dfrac{r_{-1}\gamma_i}{r_{i-1}r_i}+\dfrac{r_{-1}}{r_k}\left(\dfrac{\gamma_{k+1}}{r_{k+1}}+\dfrac{r_k}{r_{k+1}}(\gamma_{k+2}(r_{k+2}^{-1})_{r_{k+1}})_{r_{k+1}}\right)\\
=\displaystyle\sum_{i=0}^{k+1} \dfrac{r_{-1}\gamma_i}{r_{i-1}r_i}+\dfrac{r_{-1}}{r_{k+1}}(\gamma_{k+2}(r_{k+2}^{-1})_{r_{k+1}})_{r_{k+1}}$.\\
Hence, Eq. \eqref{eq:extmodinvform3} is also true for $n=k+1$, provided that it is true for $n=k$.
\item
If $\gamma_{n+1}=0$, or $r_n=1$, then $(\gamma_{n+1}(r_{n+1}^{-1})_{r_{n}})_{r_{n}}=0$, and hence, Eq.\eqref{eq:extmodinvform1} is obtained.
\end{enumerate}
\item\
Assume $r_i\ne1$ and $\gamma_{i+1}\ne0$ for all $i=0, 1, 2, \ldots, n-1$, then we can prove, by induction that
\begin{equation}\label{eq:extmodinvform4}
(\gamma_0(r_0^{-1})_{r_{-1}})_{r_{-1}}=\displaystyle\sum_{i=0}^{n}f_i\beta_i-s_{n+1}f_n\dfrac{\gamma_{n+1}}{r_n}+\dfrac{r_{-1}}{r_n}(\gamma_{n+1}(r_{n+1}^{-1})_{r_{n}})_{r_{n}}
\end{equation}
\begin{enumerate}[a.]
\item
First we want to prove
\begin{equation*}
r_{-1}=f_i r_{i-1}+s_i f_{i-1}r_i
\end{equation*}
\begin{enumerate}[i.)]
\item
When $i=1, r_{-1}=f_1 r_0+s_1 f_0 r_1=c_1 r_0+s_1 r_1$.  Hence, it is true.
\item
Assume it is true for $i=k$, then\\
$r_{-1}=f_k r_{k-1}+s_k f_{k-1}r_k=f_k(c_{k+1}r_k+s_{k+1}r_{k+1})+s_k f_{k-1}r_k=(c_{k+1}f_k+s_k f_{k-1})r_k+s_{k+1}f_i r_{k+1}=f_{k+1}r_k+s_{k+1}f_k r_{k+1}$.
Hence, it is true for $i=k+1$.
\end{enumerate}
\item
Since $(\gamma_0(r_0^{-1})_{r_{-1}})_{r_{-1}}=\dfrac{\gamma_0}{r_0}+\dfrac{r_{-1}}{r_0}(\gamma_1(r_1^{-1})_{r_0})_{r_0}\\
=\dfrac{\beta_1r_0-s_1\gamma_1}{r_0}+\dfrac{r_{-1}}{r_0}(\gamma_1(r_1^{-1})_{r_0})_{r_0}
=\beta_1-s_1 f_0\dfrac{\gamma_1}{r_0}+\dfrac{r_{-1}}{r_0}(\gamma_1(r_1^{-1})_{r_0})_{r_0}$,
therefore, Eq. \eqref{eq:extmodinvform4} is true when $n=0$.
\item
If it is true for $n=k-1$, then\\
$(\gamma_0(r_0^{-1})_{r_{-1}})_{r_{-1}}
=\displaystyle\sum_{i=0}^{k-1}f_i\beta_i-s_{k}f_{k-1}\dfrac{\gamma_{k}}{r_{k-1}}+\dfrac{r_{-1}}{r_{k-1}}(\gamma_{k}(r_{k}^{-1})_{r_{k-1}})_{r_{k-1}}\\
=\displaystyle\sum_{i=0}^{k-1}f_i\beta_i-s_{k}f_{k-1}\dfrac{\gamma_{k}}{r_{k-1}}+\dfrac{r_{-1}}{r_{k-1}}\left(\dfrac{\gamma_k}{r_k}+(\gamma_{k+1}(r_{k+1}^{-1})_{r_{k}})_{r_{k}}\right)\\
=\displaystyle\sum_{i=0}^{k-1}f_i\beta_i-s_{k}f_{k-1}\dfrac{\gamma_{k}}{r_{k-1}}+\dfrac{f_k r_{k-1}+s_k f_{k-1}r_k}{r_{k-1}}\dfrac{\gamma_k}{r_k}+\dfrac{r_{-1}}{r_{k-1}}(\gamma_{k+1}(r_{k+1}^{-1})_{r_{k}})_{r_{k}}\\
=\displaystyle\sum_{i=0}^{k-1}f_i\beta_i+f_k\dfrac{\beta_{k+1}r_k-s_{k+1}\gamma_{k+1}}{r_k}+\dfrac{r_{-1}}{r_{k-1}}(\gamma_{k+1}(r_{k+1}^{-1})_{r_{k}})_{r_{k}}\\
=\displaystyle\sum_{i=0}^{k}f_i\beta_{i}-s_{k+1}f_k\dfrac{\gamma_{k+1}}{r_k}+\dfrac{r_{-1}}{r_{k-1}}(\gamma_{k+1}(r_{k+1}^{-1})_{r_{k}})_{r_{k}}$
\\Hence, it is true for $n=k$.
\item
Since $r_n=1 \implies \gamma_{n+1}=0$, $(\gamma_0(r_0^{-1})_{r_{-1}})_{r_{-1}}=\displaystyle\sum_{i=0}^{n}f_i\beta_i$ when $r_n=1$ or $\gamma_{n+1}=0$.  Hence, Eq.\eqref{eq:extmodinvform2} is proved.
\end{enumerate}
\end{enumerate}
\end{proof}

\begin{corollary}
For the algorithm in Theorem \ref{th:theorem0}, given $s_i\in\{-1,1\}$, then $r_i, c_i, \beta_i, \gamma_i$ are generated by the following equations :
\begin{equation}\label{eq:extinvalgo1}
\left.\begin{aligned}
c_{i+1} &=s_{i+1}\Bigl \lfloor s_{i+1}\dfrac{r_{i-1}}{r_i} \Bigl \rfloor,\\
r_{i+1}&=(s_{i+1}\,r_{i-1})_{r_i},\\
\gamma_{i+1}&=(-s_{i+1}\gamma_i)_{r_i},\\
\beta_i&=s_{i+1}\Bigl \lceil s_{i+1}\dfrac{\gamma_i}{r_i} \Bigl \rceil\\
f_i&=c_if_{i-1}+s_{i-1}f_{i-2}
\end{aligned}\right\}\text{ when } r_i>1
\end{equation}
When $r_i$=1, $c_{i+1}=r_{i-1}, r_{i+1}=0,\gamma_{i+1}=0,\beta_i=\gamma_i$.
\end{corollary}

\begin{proof}
The algorithm in Theorem \ref{th:theorem0} requires $r_i, c_i, \beta_i, \gamma_i$ to satisfy 
\begin{equation*}
r_{i-1}=c_{i+1}\,r_i+s_{i+1}\,r_{i+1}, \gamma_i=\beta_i\,r_i-s_{i+1}\,\gamma_{i+1}
\end{equation*}
When $r_i=1$, these conditions are obviously satisfied.  When $r_i>1$, these conditions are also satisfied with Eq.\eqref{eq:extinvalgo1}, by the following substitutions.

\begin{enumerate}[1.]
\item $r_i$
\begin{flalign}
c_i\,r_{i-1}+s_i\,r_i&=\left(s_i\Bigl \lfloor s_i\dfrac{r_{i-2}}{r_{i-1}} \Bigl \rfloor\right)r_{i-1}+s_i(s_i\,r_{i-2})_{r_{i-1}}\notag\\
&=s_i\left(\Bigl \lfloor s_i\dfrac{r_{i-2}}{r_{i-1}} \Bigl \rfloor r_{i-1}+(s_i\,r_{i-2})_{r_{i-1}}\right)=s_i(s_i\,r_{i-2})=r_{i-2}\notag
\end{flalign}

\item $\gamma_i$
\begin{flalign}
\beta_i\,r_i-s_{i+1}\,\gamma_{i+1}&=s_{i+1}\Bigl \lceil s_{i+1}\dfrac{\gamma_i}{r_i} \Bigl \rceil\,r_i-s_{i+1}(-s_{i+1}\gamma_i)_{r_i}\notag\\
&=-s_{i+1}\Bigl \lfloor -s_{i+1}\dfrac{\gamma_i}{r_i} \Bigl \rfloor\,r_i-s_{i+1}(-s_{i+1}\gamma_i)_{r_i}\notag\\
&=-s_{i+1}\left(\Bigl \lfloor -s_{i+1}\dfrac{\gamma_i}{r_i} \Bigl \rfloor\,r_i+(-s_{i+1}\gamma_i)_{r_i}\right)\notag\\
&=-s_{i+1}(-s_{i+1}\gamma_i)=\gamma_i\notag
\end{flalign}
\end{enumerate}
\end{proof}

Alternately, Eq. \eqref{eq:extinvalgo1} can be written as :
\begin{equation}\label{eq:extinvalgo2}
\left.\begin{aligned}
c_i &=s_i\Bigl \lfloor s_i\dfrac{r_{i-2}}{r_{i-1}} \Bigl \rfloor,\\
r_i&=(s_i\,r_{i-2})_{r_{i-1}},\\
\gamma_i&=(-s_i\gamma_{i-1})_{r_{i-1}},\\
\beta_i&=s_{i+1}\Bigl \lceil s_{i+1}\dfrac{\gamma_i}{r_i} \Bigl \rceil\\
f_i&=c_if_{i-1}+s_{i-1}f_{i-2}
\end{aligned}\right\}\text{ when } r_i>0
\end{equation}

\begin{corollary}\label{th:corol1}
In Theorem \ref{th:theorem0}, if $\gcd(p,q)=d\ne1$, and with the same algorithm to generate  $r_{i-1}=c_{i+1}\,r_i+s_{i+1}\,r_{i+1}, \gamma_i=\beta_i\,r_i-s_{i+1}\,\gamma_{i+1}, f_i=c_if_{i-1}+s_{i-1}f_{i-2}$,
then 
\begin{enumerate}[1.]
\item The series $r_i$ will terminate at $r_{n+1}=0$, or $\gamma_{n+1}=0$.
\item If $r_{n+1}=0$, then $r_n=d=\gcd(p,q)$, and $d|r_i$ for all $i=-1,0,1, \ldots n$.
\item If $r_{n+1}=0 \,\&\, r_n \nmid \gamma_{n+1}$, then $d\nmid a$, and there is no solution for $(a(q^{-1})_p)_p$.
\item If $r_{n+1}=0$ and $\gamma_{m+1}=0, m\le n$, then
\begin{equation*}
 \left(\left(\dfrac{a}{d}\right)\left(\dfrac{q}{d}\right)^{-1}_{\frac{p}{d}}\right)_{\frac{p}{d}}=\displaystyle\sum_{i=0}^{m} \dfrac{p\,\gamma_i}{r_{i-1}r_i}=\displaystyle\sum_{i=0}^{m} f_i\beta_i 
\end{equation*}
\end{enumerate}
\end{corollary}

\begin{proof}
\begin{enumerate}[1.]
\item Since $r_i$ is a decreasing series of positive integers, it will eventually terminate at 0. If $\gamma_m$ reaches 0 before $r_n$, where $m<n$, then $\gamma_i=0$ for $i=m,m+1, \ldots, n$
\item Let $r_{n+1}=0$, then
\begin{flalign}
\gcd(r_{n-1},r_n)&=\gcd(c_{n+1}\,r_n+s_{n+1}\,r_{n+1},r_n)=\gcd(c_{i+1}\,r_n,r_n)\notag\\
&=r_n=d\notag.
\end{flalign}
For $i=0, 1, \ldots, n-1$,
\begin{flalign}
\gcd(r_{i-1},r_i)&=\gcd(c_{i+1}\,r_i+s_{i+1}\,r_{i+1},r_i)=\gcd(s_{i+1}\,r_{i+1},r_i)\notag\\
&=\gcd(r_{i+1},r_i)=\gcd(r_i,r_{i+1})\notag
\end{flalign}
Hence, $\gcd(p,q)=\gcd(r_{-1},r_0)=\gcd(r_{n-1},r_n)=d$, and $d|r_i$.
\item For $i=0, 1, \ldots, n, \gamma_{i+1}=s_{i+1}(\beta_i\,r_i-\gamma_i) \implies \gamma_{i+1}=(-s_{i+1}\gamma_i)_d$.\\
Hence, $\gamma_{n+1}=((-1)^{n+1}s_1s_2\ldots s_{n+1}\gamma_0)_d$.\\
Since $d\nmid\gamma_{n+1}$, therefore $d\nmid \gamma_0$, i.e., $d\nmid a$.
\item Assume $r_{n+1}=0$. Then, $r_n=1$, or $r_n\ne1$.
\begin{enumerate}[a.)]
\item $r_n=1$\\
$\gcd(p,q)=1$, and therefore the Theorem \ref{th:theorem0} applies.
\item $r_n =d\ne1$. Assume $d|\gamma_{n+1}$, otherwise as seen above, $d\nmid a$ and there will be no solution.
Then let $\gamma_{m+1}=0$, where $m\le n$ and $\gamma_i\ne 0$, for $i=0,1,\ldots,m$.  Therefore, the equations in Theorem \ref{th:theorem0} will terminate at $i=m$.
\end{enumerate}
\end{enumerate}
\end{proof}

The algorithm for extended modular inverse detailed in Theorem \ref{th:theorem0} exhibits some flexibilities which can be exploited for optimization for various specific types of problems.
For example, if $a=1$ and all $s_i$'s are 1, then $\beta_i=1$ and the formula to calculate $f_i$'s is equivalent to $j_s=q_sj_{s-1}+j_{s-2}$ in \cite[p.20]{ding+pei+salomaa}, which is the Dayan Qiuyi Shu.
For this reason, Theorem \ref{th:theorem0} may be called a general Dayan algorithm.
Dayan algorithm has been known in China since the 13th century or earlier. However, except some brief mention of it in a few literatures such as \cite{ding+pei+salomaa}, the Dayan algorithm is rarely discussed in details in the English literature other than from a historical point of view.

Obviously, when $a=1$, this algorithm produces the ordinary modular inverse, i.e., $(q^{-1})_p$.

One may also notice a property of this algorithm is that it terminates when $\gamma_{n+1}=0$, even if $r_n\ne1$.
That means for some $a$'s, it does not require to perform the division operation until $r_n=1$.
To one extreme, for example, if $a=k\,q<p$, then Eq.\eqref{eq:extreciprocity} provides the answer as $k$ after the first division step when $\gamma_1=0$.  This is also apparent from the following equation :
\begin{equation}
(k\,q(q^{-1})_p)_p =(k(1+p\,q-p(p^{-1})_q))_p=k\notag
\end{equation}

For $a$'s of the form $a=k_1\,q-k_2\,r_1$, where $k_2 r_1<q$, the algorithm will terminate after the second division step, when $\gamma_2=0$.
$a$'s of other forms may also result in the algorithm to terminate before $r_n=1$.
Examples to illustrate this property will be shown at the end of this paper.

Note the flexibility in setting up $s_i$'s, one may also select each $s_i$ independently to produce $r_i$ to have the least absolute value. For the Euclidean Algorithm, it is known as the Method of the Least Absolute Remainder.

Among all different possible variations and flavours of this algorithm, some of the applications will be discussed in slightly greater details in the next section.

\section{Applications}

\subsection{Modular Inverse Formula}
\begin{corollary}\emph{(Modular Inverse Formula, first type)}
Let $p,q$ be positive integers, $1<q<p$ and co-prime to each other, the modular inverse is given by 
\begin{flalign}
\label{eq:modinvform1}
(q^{-1})_p&=p\sum_{i=0}^{n}\dfrac{(-1)^i}{r_{i-1}r_i}+\begin{cases} 0 & n=even \\p & n=odd \end{cases}\\
\label{eq:modinvform2}
(q^{-1})_p&=-p\sum_{i=0}^{\lfloor\frac{n+1}{2}\rfloor}\dfrac{c_{2i+1}}{r_{2i-1}r_{2i}}+\begin{cases} \dfrac{1}{r_{n-1}} & n=even \\p & n=odd \end{cases}
\end{flalign}
where $r_i=(r_{i-2})_{r_{i-1}}$ and $r_n=1$.
\end{corollary}

\begin{proof}
\begin{enumerate}[1.]
\item Equation\eqref{eq:modinvform1}\\
In Theorem \ref{th:theorem0}, let $a=1, s_i=-1$, then from Eq.\eqref{eq:extinvalgo1}, we obtain
$r_i=(r_{i-2})_{r_{i-1}}$ and $\gamma_0=1, \gamma_i=(-\gamma_{i-1})_{r_{i-1}}$.
\begin{enumerate}[i.)]
\item
First, we have to show $\gamma_{i}=r_{i-1}-r_i+(-1)^i$, for $i= 2,3, \ldots, n-1, and\\ \gamma_n=\begin{cases} (-1)^n-r_n & n=even\\r_{n-1}-r_n+(-1)^n & n=odd \end{cases}$.
\begin{enumerate}[(a)]
\item
$\gamma_1=(-1)_{r_0}=r_0-1$,\\
$\gamma_2=(-(r_0-1))_{r_1}=(1-(r_0)_{r_1})_{r_1}=(1-r_2))_{r_1}=r_1-r_2+1$\\
$\gamma_3=(-(r_1-r_2+1))_{r_2}=(-r_1-1)_{r_2}=(-r_3-1)_{r_2}=r_2-r_3-1$
\item
If it is true for $i=k$, then
$\gamma_{k+1}=(-\gamma_k)_{r_k}=(-(r_{k-1}-r_k+(-1)^k))_{r_k}=(-r_{k-1}-(-1)^k)_{r_k}=(-r_{k+1}+(-1)^{k+1})_{r_k}=r_k-r_{k+1}+(-1)^{k+1}$, since $-r_{k+1}+(-1)^{k+1}<0$.\\
Hence, this equation is true for $i=2,3, \ldots, n-1.$
\item
When $r_n=1, \gamma_n=(-\gamma_{n-1})_{r_{n-1}}=(-r_{n-2}+r_{n-1}-(-1)^{n-1})_{r_{n-1}}=(-r_{n-2}+1)_{r_{n-1}}=(-r_n+1)_{r_{n-1}}=0$, if $n=$even.
If $n=$odd, then $\gamma_n=(-r_n-1)_{r_{n-1}}=r_{n-1}-r_n-1$, since $-r_n-1<0$.
\end{enumerate}
\item
\begin{flalign}
\displaystyle\sum_{i=0}^{n}\dfrac{\gamma_i}{r_{i-1}r_i}&=\dfrac{1}{r_{-1}r_0}+\dfrac{r_0-1}{r_0 r_1}+\displaystyle\sum_{i=2}^{n}\dfrac{\gamma_i}{r_{i-1}r_i}\notag\\
&=\dfrac{1}{r_{-1}r_0}+\dfrac{r_0-1}{r_0 r_1}+\displaystyle\sum_{i=2}^{n-1}\dfrac{r_{i-1}-r_i+(-1)^i}{r_{i-1}r_i}+\dfrac{\gamma_n}{r_{n-1}r_n}\notag\\
&=\dfrac{1}{r_{-1}r_0}+\dfrac{1}{r_1}-\dfrac{1}{r_0 r_1}+\displaystyle\sum_{i=2}^{n-1}\left( \dfrac{1}{r_i} - \dfrac{1}{r_{i-1}} + \dfrac{(-1)^i}{r_{i-1}r_i} \right)+\dfrac{\gamma_n}{r_{n-1}r_n}\notag\\
&=\displaystyle\sum_{i=0}^{n-1}\dfrac{(-1)^i}{r_{i-1}r_i}+\dfrac{1}{r_{n-1}}+\dfrac{\gamma_n}{r_{n-1}r_n}\notag\\
&=\displaystyle\sum_{i=0}^{n-1}\dfrac{(-1)^i}{r_{i-1}r_i}+\dfrac{1}{r_{n-1}}+\begin{cases} \dfrac{(-1)^n-r_n}{r_{n-1}r_n} & n=even \\ \dfrac{r_{n-1}-r_n+(-1)^n}{r_{n-1}r_n} & n=odd \end{cases}\notag\\
&=\displaystyle\sum_{i=0}^{n}\dfrac{(-1)^i}{r_{i-1}r_i}+\begin{cases} 0 & n=even \\
1 & n=odd \end{cases}\notag
\end{flalign}
\end{enumerate}
\item Equation \eqref{eq:modinvform2}\\
Since 
\begin{flalign}
\displaystyle\sum_{i=0}^{2\lfloor \frac{n-1}{2} \rfloor+1}\dfrac{(-1)^i}{r_{i-1}r_i}&=\displaystyle\sum_{i=0}^{\lfloor\frac{n-1}{2}\rfloor}\left(\dfrac{(-1)^{2i}}{r_{2i-1}r_{2i}}+\dfrac{(-1)^{2i+1}}{r_{2i}r_{2i+1}} \right) \notag \\
&=\displaystyle\sum_{i=0}^{\lfloor\frac{n-1}{2}\rfloor}\left(\dfrac{(-1)^{2i}(r_{2i+1}-r_{2i-1})}{r_{2i-1}r_{2i}r_{2i+1}} \right) 
=\displaystyle\sum_{i=0}^{\lfloor\frac{n-1}{2}\rfloor}\dfrac{-c_{2i+1}}{r_{2i-1}r_{2i+1}} \notag
\end{flalign}
, then
$\displaystyle\sum_{i=0}^{n}\dfrac{(-1)^i}{r_{i-1}r_i}=\displaystyle\sum_{i=0}^{\lfloor\frac{n-1}{2}\rfloor}\dfrac{-c_{2i+1}}{r_{2i-1}r_{2i+1}} +\begin{cases} \dfrac{1}{r_{n-1}r_n} & n=even \\ 0 & n=odd\end{cases}$
\\Combining this with Eq.\eqref{eq:modinvform1}, and Eq.\eqref{eq:modinvform2} will be obtained.
\end{enumerate}
\end{proof}
In order not to keep track of the parity of $n$, we may simply calculate the value of $p\displaystyle\sum_{i=0}^{n}\dfrac{(-1)^i}{r_{i-1}r_i}$,
if it is negative, then add $p$ to obtain the modular inverse.
Eq. \eqref{eq:modinvform1} can also be obtained through the Reciprocity formula for ordinary modular inverse, since 
\begin{equation}
\dfrac{(q^{-1})_p}{p}+\dfrac{(p^{-1})_q}{q}=1+\dfrac{1}{p\,q}\notag
\end{equation}

\begin{corollary}\emph{(Modular Inverse Formula, sceond type)}
Let $p,q$ be positive integers, $1<q<p$ and co-prime to each other, the modular inverse is given by 
\begin{flalign}
\label{eq:modinvform3}
(q^{-1})_p&=p\sum_{i=0}^{n}\dfrac{1}{r_{i-1}r_i} \\
\label{eq:modinvform4}
(q^{-1})_p&=p\sum_{i=0}^{\lfloor\frac{n+1}{2}\rfloor}\dfrac{c_{2i+1}}{r_{2i-1}r_{2i}}+\begin{cases} \dfrac{1}{r_{n-1}} & n=even \\0 & n=odd \end{cases}
\end{flalign}
where $r_i=(-r_{i-2})_{r_{i-1}}$ and $r_n=1$.
\end{corollary}

\begin{proof}
\begin{enumerate}[1.)]
\item Equation \eqref{eq:modinvform3}\\
In Theorem \ref{th:theorem0}, let $a=1, s_i=-1$, then from Eq.\eqref{eq:extinvalgo1}, we obtain
$r_i=(-r_{i-2})_{r_{i-1}}$ and $\gamma_i=(\gamma_{i-1})_{r_{i-1}}$. Since $\gamma_0=1$, we have $\gamma_i=1$ for all $i$'s.
Hence, the equation is proved.
\item Equation \eqref{eq:modinvform4}\\
Note that $\dfrac{r_{i-1}+r_{i+1}}{r_{i-1}r_{i}r_{i+1}}=\dfrac{c_{i+1}}{r_{i-1}r_{i+1}}$ and similar to the proof for Eq. \eqref{eq:modinvform2}, it is straightforward to prove Eq.\eqref{eq:modinvform4}.
\end{enumerate}
\end{proof}

With Eqs.\eqref{eq:modinvform1} and \eqref{eq:modinvform3}, it is necessary to keep track of only all the remainders and no quotients are needed.
Probably, they are the shortest or simplest formula (algorithm) for calculating modular inverse though the trade-off is that floating point divisions are used.
However, due to the granularity of floating point numbers, this may be useful for small numbers only.
Furthermore, at a first glance, it seems to suggest the Eq.\eqref{eq:modinvform3} is a better choice to calculate the modular inverse, since all the terms are positive and there is no need to keep track of the parity of $n$.
However, expanding $\dfrac{p}{q}$ into continued fractions will show that it will generally take more division steps for this formula.

\begin{corollary}
Let $p,q$ be positive integers, $1<q<p$ and co-prime to each other, if $r_i=(-r_{i-2})_{r_{i-1}}$ and $r_n=1$ then for $j=0,1,\ldots,n$
\begin{equation}
\label{eq:modinvcal}
(r_j(q^{-1})_p)_p=f_j
\end{equation}
\end{corollary}
\begin{proof}
In Theorem \ref{th:theorem0}, let $s_i=-1$.  Then, $\gamma_0=r_j, \gamma_i=(\gamma_{i-1})_{r_{i-1}}$ and $\beta_i=-\Bigl \lceil -\dfrac{\gamma_i}{r_{i-1}} \Bigl \rceil = \Bigl \lfloor \dfrac{\gamma_i}{r_{i-1}} \Bigl \rfloor$.
Since $r_j<r_{j-1}< \ldots < r_1<r_0=q$, then $\beta_0=\beta_1=\ldots=\beta_{j-1}=0, \beta_j=1$, and $\beta_{j+1}=\ldots=\beta_n=0$.
Hence, from Eq. \eqref{eq:extmodinvform2},\\
\centerline{$(r_j(q^{-1})_p)_p=\displaystyle\sum_{i=0}^{n} f_i\beta_i=f_j$}
\end{proof}
Furthermore, from Eq.\eqref{eq:modinvcal}, it can easily be deduced that, for $k\in\mathbb{Z}$,
\begin{equation}
(k\,r_j(q^{-1})_p)_p=(k\,f_j)_p\notag
\end{equation}

\subsection{Chinese Remainder Algorithm}
The next example demonstrates the application of extended modular inverse to the linear congruence equations, i.e., the Chinese Remainder Theorem.
This algorithm is the $n$-congruence iterative CRA from \cite[p.23]{ding+pei+salomaa}, and it is also known as Garner's formula.

\begin{theorem}\label{th:CRT}\emph{(Chinese Remainder Theorem)}
Let $m_1,m_2, \ldots, m_n$ be pair-wise relatively prime and greater than 1, and $a_1, a_2, \ldots$ , $a_n$ are integers less than $m_1, m_2, \ldots , m_n$ respectively.

\begin{equation}\label{eq:CRTeq}
\left.\begin{aligned}
x&\equiv a_1 \pmod{m_1}\\
x&\equiv a_2 \pmod{m_2}\\
&\ldots \ldots\\
&\ldots \ldots\\
x&\equiv a_n \pmod{m_n}
\end{aligned}\right\}
\end{equation}

The solution of Eq.\eqref{eq:CRTeq} is given by, $x=x_n$ from the following iterative algorithm :
\begin{equation}
\left.\begin{aligned}
x_1&=a_1\\
x_2&=m_1((a_2-x_1)(m_{1}^{-1})_{m_2})_{m_2}+x_1\\
&\ldots \ldots\\
&\ldots \ldots\\
x_n&=m_1m_2 \ldots m_{n-1}((a_n-x_{n-1})((m_1m_2 \ldots m_{n-1})^{-1})_{m_n})_{m_n}+x_{n-1}
\end{aligned}\right\}\notag
\end{equation}
\end{theorem}

\begin{proof}
It suffices to shows that $x_2$ satisfies the first two congruence equations.\\
It is obvious that $x_2$ satisfies the first congruence equation,
and we need to show the second congruence equation is also satisfied.

\begin{flalign}
(x_2)_{m_2}&=(m_1((a_2-x_1)(m_{1}^{-1})_{m_2})_{m_2}+x_1)_{m_2}\notag\\
&=((m_1(a_2-x_1)(m_{1}^{-1})_{m_2})_{m_2}+x_1)_{m_2}\notag\\
&=(((a_2-x_1)(1+m_1 m_2-m_2(m_{2}^{-1})_{m_1}))_{m_2}+x_1)_{m_2}\notag\\
&=((a_2-x_1)+x_1)_{m_2}=(a_2)_{m_2}=a_2\notag
\end{flalign}
\end{proof}
This algorithm requires only $(n-1)$ extended modular inverses, and more importantly, extended modular inverse calculation statistically requires less division steps,
though detailed efficiency analysis still need to be studied.

\begin{theorem}\emph{(Chinese Remainder Theorem, non-co-prime moduli)}\\
Let $m_1,m_2$ be integers greater than 1, not necessarily co-prime to each other, and $a_1, a_2$ are integers less than $m_1, m_2$ respectively.
\begin{equation*}
\left.\begin{aligned}
x&\equiv a_1 \pmod{m_1}\\
x&\equiv a_2 \pmod{m_2}
\end{aligned}\right\}\notag
\end{equation*}
Let $d=\gcd(m_1,m_2)$, and $d|(a_1-a_2)$, then, the solution is given by :
\begin{equation}\label{eq:CRTsol}
x_2=m_1((a_2-a_1)(m_1^{-1})_{m_2})_{m_2}+a_1+k\dfrac{m_1 m_2}{d}
\end{equation}
where $k\in\mathbb{Z}$ and by Corollary \ref{th:corol1}, with $a=(a_2-a_1)_{m_2}$,\\
\centerline{$((a_2-a_1)(m_{1}^{-1})_{m_2})_{m_2}=\left(\left(\dfrac{a}{d}\right)\left(\dfrac{m_1}{d}\right)^{-1}_{\frac{m_2}{d}}\right)_{\frac{m_2}{d}}$.}
\end{theorem}

\begin{proof}
This solution is obviously true for co-prime moduli, by Theorem \ref{th:CRT}.\\
Let's assume $\gcd(m_1,m_2)=d\ne1$ and $d|a$.\\
It is obvious to see $(x_2)_{m_1}=a_1$.
\begin{flalign}
(x_2)_{m_2}&=\left(m_1(a(m_{1}^{-1})_{m_2})_{m_2}+a_1+k\dfrac{m_1 m_2}{d}\right)_{m_2}\notag\\
&=\left(d\dfrac{m_1}{d}\left(\left(\dfrac{a}{d}\right)\left(\dfrac{m_1}{d}\right)^{-1}_{\frac{m_2}{d}}\right)_{\frac{m_2}{d}}+a_1\right)_{m_2}
=\left(d\left(\dfrac{a}{d}+\dfrac{m_2}{d}k_2\right)+a_1\right)_{m_2}\notag\\
&=((a_2-a_1)_{m_2}+a_1)_{m_2}=(a_2)_{m_2}\notag
\end{flalign}
\end{proof}

With Eq.\eqref{eq:CRTsol} as the solution of a two-modulus Chinese Remainder Algorithm, it is possible to parallelize the algorithm for improved efficiency.
For example, there is a set of 16 simultaneous linear congruence equations, not necessarily all co-prime to each other.
We may first calculate the 8 solutions for 8 sets of two-modulus simultaneous linear congruence equations, in parallel independently.
With these 8 solutions, they can be processed by 4 parallel threads or processors, and then two threads in the next round.
Thus, in general, it takes $\lceil\operatorname{Log}_2(n)\rceil$ rounds of parallel computations of a two-modulus Chinese Remainder Algorithm.

\section{Examples}

\begin{table}[t]
\centering
\begin{tabular}{ | r | r | r | r | r  | r | r | r | r | r |}    \hline
    Step $i$                                   &  -1 &   0 &   1 &   2 &  3 & 4  &  5 &  6 &  7   \\ \hline
\hline
$r_i=(r_{i-2})_{r_{i-1}}$                      & 189 & 106 &  83 &  23 & 14 &  9 &  5 &  4 &  1   \\ \hline
$\gamma_i=(\gamma_{i-1})_{r_{i-1}}$            &     &   1 & 105 &  61 &  8 &  6 &  3 &  2 &  2   \\ \hline
$c_i=\lfloor \frac{r_{i-2}}{r_{i-1}}\rfloor$   &     &     &   1 &  1  &  3 &  1 &  1 &  1 &  1   \\ \hline
$\beta_i=\lceil\frac{\gamma_i}{r_i}\rceil$     &     &   1 &   2 &  3  &  1 &  1 &  1 &  1 &  2   \\ \hline
$f_i=c_i f_{i-1}+f_{i-2}$                      &   0 &   1 &   1 &  2  &  7 &  9 & 16 & 25 & 41   \\ \hline
\end{tabular}
\caption{Illustration for Extended Modular Inverse Algorithm to find out $(106^{-1})_{189}$ with $s_i=1$}
\label{tab:XMIAlg}
\end{table}

Here are some examples to demonstrate the use of some algorithms and formulas discussed in this paper.

Table \ref{tab:XMIAlg} gives the results generated by the Extended Modular Inverse Algorithm detailed in Theorem \ref{th:theorem0}.

\begin{example} Equation \eqref{eq:extmodinvform2}
\begin{flalign}
(106^{-1})_{189}&=\displaystyle\sum_{i=0}^{n} f_i\beta_i \notag\\
&=1 \cdot 1+1\cdot2+2\cdot3+7\cdot1+9\cdot1+16\cdot1+25\cdot1+41\cdot2\notag\\
&=148\notag
\end{flalign}

\end{example}

\begin{example} Equation \eqref{eq:extmodinvform1}
\begin{flalign}
(106^{-1})_{189}&=p\displaystyle\sum_{i=0}^{n} \dfrac{\gamma_i}{r_{i-1}r_i}\notag\\
&=189(\dfrac{1}{189\cdot106}+\dfrac{105}{106\cdot83}+\dfrac{61}{83\cdot23}+\dfrac{8}{23\cdot14}+\dfrac{6}{14\cdot9}\notag\\&+\dfrac{3}{9\cdot5}+\dfrac{2}{5\cdot4}+\dfrac{2}{4\cdot1})\notag\\
&=148\notag
\end{flalign}

\end{example}

\begin{example} Equation \eqref{eq:modinvform1}
\begin{flalign}
(106^{-1})_{189}&=p\displaystyle\sum_{i=0}^{n} \dfrac{(-1)^i}{r_{i-1}r_i}\notag\\
&=189(\dfrac{1}{189\cdot106}-\dfrac{1}{106\cdot83}+\dfrac{1}{83\cdot23}-\dfrac{1}{23\cdot14}+\dfrac{1}{14\cdot9}\notag\\&-\dfrac{1}{9\cdot5}+\dfrac{1}{5\cdot4}-\dfrac{1}{4\cdot1})\notag\\
&=-41\notag
\end{flalign}
Since it gives a negative value, and also note $n=7=odd$, and therefore we have to add $p$ to it, and hence,
$(106^{-1})_{189}=-41+189=148$.

\end{example}

\begin{example} Non-identical $s_i$'s\\
Table \ref{tab:XMImode0} shows the results for extended modular inverse calculation with different $s_i$ in each step.

\begin{flalign}
(106^{-1})_{189}&=\displaystyle\sum_{i=0}^{n} f_i\beta_i \notag\\
&=1 \cdot 0+2\cdot0+9\cdot0+25\cdot1+41\cdot3\notag\\
&=148\notag\\
(106^{-1})_{189}&=189\left(\dfrac{1}{189\cdot106}+\dfrac{1}{106\cdot23}+\dfrac{1}{23\cdot9}+\dfrac{1}{9\cdot4}+\dfrac{3}{4\cdot1}\right)\notag\\
&=148\notag
\end{flalign}

\begin{table}[t]
\centering
\begin{tabular}{ | r | r | r | r | r  | r | r | }    \hline
    Step $i$                                                 &  -1 &   0 &   1 &   2 &  3 &  4         \\ \hline
\hline
$r_i=(s_i\,r_{i-2})_{r_{i-1}}$                               & 189 & 106 &  23 &   9 &  4 &  1         \\ \hline
$\gamma_i=(-s_i\,\gamma_{i-1})_{r_{i-1}}$                    &     &   1 &   1 &   1 &  1 &  3         \\ \hline
$s_i$                                                        &     &     &  -1 &  -1 & -1 &  1         \\ \hline
$c_i=s_i\lfloor s_i\frac{r_{i-2}}{r_{i-1}}\rfloor$           &     &     &   2 &   5 &  3 &  2         \\ \hline
$\beta_i=s_{i+1}\lceil s_{i+1}\frac{\gamma_i}{r_i}\rceil$    &     &   0 &   0 &   0 &  1 &  3         \\ \hline
$f_i=c_i f_{i-1}+s_{i-1}f_{i-2}$                             &   0 &   1 &   2 &   9 & 25 & 41         \\ \hline
\end{tabular}
\caption{Illustration for Extended Modular Inverse Algorithm to find out $(106^{-1})_{189}$ with varying $s_i$'s.}
\label{tab:XMImode0}
\end{table}

\end{example}

\begin{example} Chinese Remainder Algorithm\\
We are going to solve the following simultaneous congruence equations :
\begin{equation*}
\left.\begin{aligned}
x&\equiv  5 &\pmod{106}\\
x&\equiv 51 &\pmod{189}
\end{aligned}\right\}\notag
\end{equation*}

\begin{table}[t]
\centering
\begin{tabular}{ | r | r | r | r | r  | r | r | r | r | r |}    \hline
    Step $i$                                &  -1 &   0 &   1 &   2 &  3 & 4  &  5 &  6 &  7   \\ \hline
\hline
$r_i=(r_{i-2})_{r_{i-1}}$                   & 189 & 106 &  83 &  23 & 14 &  9 &  5 &  4 &  1   \\ \hline
$\gamma_i=(-\gamma_{i-1})_{r_{i-1}}$        &     &  46 &  60 &  23 &  0 &    &    &    &   \\ \hline
$c_i=\lfloor \frac{r_{i-2}}{r_{i-1}}\rfloor$&     &     &   1 &  1  &  3 &    &    &    &      \\ \hline
$\beta_i=\lceil\frac{\gamma_i}{r_i}\rceil$  &     &   1 &   1 &   1 &  0 &    &    &    &   \\ \hline
$f_i=c_i f_{i-1}+f_{i-2}$                   &   0 &   1 &   1 &   2 &  7 &    &    &    &   \\ \hline
\end{tabular}
\caption{Illustration for Extended Modular Inverse Algorithm to find out $(46(106^{-1})_{189})_{189}$ with $s_i=1$}
\label{tab:CRT}
\end{table}

Note we have to find the extended modular inverse $(46(106^{-1})_{189})_{189}$, by Eq.\eqref{eq:extmodinvform2}, and using the result from the Extended Modular Inverse algorithm shown in Table \ref{tab:CRT}, we get
\begin{flalign}
(46(106^{-1})_{189})_{189}&=1\cdot1+1\cdot1+2\cdot1+3\cdot0=4\notag\\
\notag\\
x&=m_1((a_2-x_1)(m_{1}^{-1})_{m_2})_{m_2}+x_1\notag\\
&=106((51-5)(106^{-1})_{189})_{189}+5=106(46(60^{-1})_{189})_{189}+5\notag\\
&=106\cdot4+5=429\notag
\end{flalign}

Comparing Tables \ref{tab:XMIAlg} and \ref{tab:CRT}, we notice that, if it is already known that the two numbers are co-prime to each other, it takes only three steps in calculating $(46(106^{-1})_{189})_{189}$,
whereas it takes 8 steps to calculate $(106^{-1})_{189}$.
If it is not known in advance whether the two moduli are co-prime to each other or not, it is still necessary to carry out the algorithm for $r_i$'s until either $1$ or $0$ is reached.

\end{example}

\begin{example} Non-co-prime Chinese Remainder Theorem\\
The following example demonstrates the efficiency of the algorithm for non-co-prime Chinese Remainder Theorem.  There is no need to pre-determine whether the two moduli are co-prime with each other, and $\gcd(m_1,m_2)$ will be found as part of the algorithm.
The solution may also be found even before GCD is found.

\begin{table}[t]
\centering
\begin{tabular}{ | r | r | r | r | r  | r | r | r | r | r || r |}    \hline
    Step $i$                                   &  -1 &   0 &   1 &   2 &   3 &   4 &  5 &  6 &  7 &  8   \\ \hline
\hline
$r_i=(r_{i-2})_{r_{i-1}}$                      & 945 & 530 & 415 & 115 &  70 &  45 & 25 & 20 &  5 &  0   \\ \hline
$\gamma_i=(-\gamma_{i-1})_{r_{i-1}}$           &     & 230 & 300 & 115 &   0 &     &    &    &    &      \\ \hline
$c_i=\lfloor \frac{r_{i-2}}{r_{i-1}}\rfloor$   &     &     &   1 &  1  &   3 &     &    &    &    &      \\ \hline
$\beta_i=\lceil\frac{\gamma_i}{r_i}\rceil$     &     &   1 &   1 &  1  &   0 &     &    &    &    &      \\ \hline
$f_i=c_i f_{i-1}+f_{i-2}$                      &   0 &   1 &   1 &  2  &   7 &     &    &    &    &      \\ \hline
\end{tabular}
\caption{Illustration for Extended Modular Inverse Algorithm to find out $(230(530^{-1})_{945})_{945}$ with $s_i=1$}
\label{tab:CRT2}
\end{table}

\begin{equation*}
\left.\begin{aligned}
x&\equiv 79 &\pmod{530}\\
x&\equiv 309 &\pmod{945}
\end{aligned}\right\}\notag
\end{equation*}

With Table \ref{tab:CRT2}, it is found that
\begin{flalign}
(230(530^{-1})_{945})_{945}&=1\cdot1+1\cdot1+2\cdot1=4\notag\\
\gcd(530,945)&=5\notag
\end{flalign}
The solution is :
\begin{flalign}
x&=530(230(530^{-1})_{945})_{945}+79+k\dfrac{530\cdot945}{5}\notag\\
&=530\cdot4+79+100170k=2199+100170k\notag
\end{flalign}
As a verification, we get\\
$(2199+100170k)_{530}=79,\\
(2199+100170k)_{945}=309$.
\end{example}

\section{Conclusions}
We have introduced a modified modular inverse and a new extended modular inverse and proved the recipocity formulas for them.
An algorithm is found to calculate the extended modular inverse, and for some specific sets of numbers, it is more efficient than the classical extended Euclidean algorithm.
It is also applicable to the Chinese Remainder problem, with co-prime moduli as well as non-co-prime moduli.

\bibliographystyle{amsplain}
\bibliography{mybib} 

\end{document}